\documentclass[11pt,a4paper,reqno]{amsart}

\usepackage[utf8]{inputenc}
\usepackage[T1]{fontenc}
\usepackage{xspace,
amsthm,
amsmath,
amssymb,
bbm,
setspace,
graphicx,
subfig,
xcolor,
keyval}
\usepackage{mathtools}

\usepackage{geometry}
\newcommand{\R}{\mathbb{R}}
\newcommand{\Rd}{\mathbb{R}^2}
\newcommand{\Rp}{\mathbb{R}^+}
\newcommand{\C}{\mathbb{C}}
\newcommand{\Id}{\mathbb{I}}

\newcommand{\G}{\mathcal{G}}

\newcommand{\F}{\mathcal{F}}
\renewcommand{\exp}{\mathrm{e}}
\renewcommand{\i}{\imath}
\newcommand{\M}{\mathrm{M}}
\newcommand{\E}{\mathrm{E}}
\renewcommand{\S}{\mathrm{S}}

\renewcommand{\Re}{\mathrm{Re}}
\renewcommand{\leq}{\leqslant}
\renewcommand{\geq}{\geqslant}

\newcommand{\la}{\lambda}
\newcommand{\La}{\Lambda}
\newcommand{\si}{\sigma}
\newcommand{\be}{\beta}
\newcommand{\al}{\alpha}
\renewcommand{\th}{\theta}
\newcommand{\om}{\omega}
\newcommand{\ga}{\gamma}

\newcommand{\x}{\mathbf{x}}

\renewcommand{\k}{\mathbf{k}}
\newcommand{\z}{\mathbf{0}}

\newcommand{\de}{\partial}
\newcommand{\lap}{\Delta}
\newcommand{\na}{\nabla}

\newcommand{\f}[2]{\frac{#1}{#2}}

\newcommand{\scal}[2]{\langle #1,#2\rangle}
\newcommand{\wt}[1]{\widetilde{#1}}
\newcommand{\ov}[1]{\overline{#1}}

\newcommand{\disp}{\displaystyle}

\theoremstyle{plain} 
\newtheorem{thm}{Theorem}[section]

\theoremstyle{definition}

\theoremstyle{definition} 
\newtheorem{defn}{Definition}[section]

\theoremstyle{remark} 
\newtheorem{rem}{Remark}[section]

\begin{document}
 
\title[Stability of the standing waves of the concentrated NLSE in 2D]{Stability of the standing waves of the \\concentrated NLSE in dimension two}

\author[R. Adami]{Riccardo Adami}
\address{Politecnico di Torino, Dipartimento di Scienze Matematiche ``G.L. Lagrange'', Corso Duca degli Abruzzi, 24, 10129, Torino, Italy.}
\email{riccardo.adami@polito.it}
\author[R. Carlone]{Raffaele Carlone}
\address{Universit\`{a} degli Studi di Napoli ``Federico II'', Dipartimento di Matematica e Applicazioni ``R. Caccioppoli'', MSA, via Cinthia, I-80126, Napoli, Italy.}
\email{raffaele.carlone@unina.it}
\author[M. Correggi]{Michele Correggi}
\address{Scuola Normale Superiore, Piazza dei Cavalieri, 7, 56126, Pisa}
\urladdr{https://sites.google.com/sns.it/michele-correggi}
\email{michele.correggi@gmail.com}
\author[L. Tentarelli]{Lorenzo Tentarelli}
\address{Politecnico di Torino, Dipartimento di Scienze Matematiche ``G.L. Lagrange'', Corso Duca degli Abruzzi, 24, 10129, Torino, Italy.}
\email{lorenzo.tentarelli@polito.it}

\date{\today}

\begin{abstract} In this paper we will continue  the analysis of two dimensional  Schr\"odinger equation  with a fixed, pointwise, nonlinearity started in \cite{ACCT-18,CCT-ANIHPC19}. In this model,  the occurrence of a blow-up phenomenon has two peculiar features:  the energy threshold under which all solutions blow up is strictly negative and coincides with the infimum of the energy of the standing waves; there is no critical power nonlinearity, i.e., for every power there exist blow-up solutions. Here we study the stability properties of stationary states to verify whether  the anomalies  mentioned before have any counterpart on the stability features.

\end{abstract}
\maketitle

\vspace{-.5cm}
{\footnotesize AMS Subject Classification: 35Q55, 35Q40.
}
\smallskip

{\footnotesize Keywords:  nonlinear Schr\"{o}dinger equation; point interactions; standing waves; orbital stability.
}

%%%%%%%%%%%%%%%%%%%%%%%%%%%%%%%%%%%%%%%%%%%%%%%%%%%%%%%%%%%%%%%%%%%%%%%%%%%%%%%%%%
%%%%%%%%%%%%%%%%%%%%%%%%%%%%%%%%%%%%%%%%%%%%%%%%%%%%%%%%%%%%%%%%%%%%%%%%%%%%%%%%%%
%%%%%%%%%%%%%%%%%%%%%%%%%%%%%%%%%%%%%%%%%%%%%%%%%%%%%%%%%%%%%%%%%%%%%%%%%%%%%%%%%%

\section{Introduction}

The Nonlinear Schr\"odinger Equation (NLSE) with concentrated nonlinearity in $d=2$ is the subject of several recent papers, finalizing a research program  developed over the last twenty years (see \cite{AT-JFA01,ADFT-ANIHPC03,ADFT-ANIHPC04,CFN-JMAA17} for the NLSE with concentrated nonlinearity and also \cite{CFT-Non19} and \cite{CCNP-SIMA17} for the fractional case and the Dirac equation, respectively). 

Such a research line was originally  motivated by some mesoscopic physical models. For instance,   in semiconductor theory  the effect of electronic charge accumulation in a resonant tunneling in a double barrier heterostructure \cite{JLPS-APHY95} is typically  studied using a concentrated NLSE.
More recently, other applications have been suggested:  the spontaneous formation of quantum coherent non-dissipative patterns in semiconductor heterostructures with nonlinear properties \cite{BKB-PRB96}; the dynamics of the mixed states of statistical physics \cite{N-Non98}; the appearance of quantum turbulence in the probability density \cite{AZ-99}; the scattering  in nuclear physics models for the disexcitation of isomeric states and also the production of weakly bounded states in heavy nuclei close to the instability; the analysis of resonant tunneling diodes, which exhibits intrinsic instability \cite{WCGBZ-03} or  the fabrication of semiconductor superlattices, for the estimate of  the time decay rates for the solutions to the Schr\"odinger-Poisson equations in the repulsive case \cite{BNR-06,SASO-04}.

In \cite{CCT-ANIHPC19} and \cite{CFT-JFA17} the local well-posedness is established, i.e., the problem of existence and uniqueness of the solution for short times, as well as the  mass  and energy conservation. Global existence is also proven in the defocusing case irrespective of the power of the nonlinearity. In \cite{ACCT-18} it is studied the occurrence of a blow-up phenomenon for a focusing nonlinearity, with two peculiar features: first, the energy threshold under which all solutions blow up is strictly negative and coincides with the infimum of the energy of standing waves; second, there is no critical power nonlinearity, i.e.,  for every power there exist blow up solutions. We remark that such a behavior is anomalous compared to the conventional NLSE, also because such anomalies are not a direct consequence of  the dimension, or of the concentrated nonlinearity. In fact, there is a critical power for standard nonlinearities in dimension two \cite{MR-AnMath05}, and there is also a critical power for concentrated nonlinearities in dimension one and three \cite{ADFT-ANIHPC03,AT-JFA01}. In the present paper we  investigate further whether  such peculiarities also show up in the stability of stationary states. 

Let us preliminary recall  the results on  the standard NLSE \cite{SUSU-99}: consider the Cauchy problem for a focusing NLSE, where the word focusing refers to the attractive character of the nonlinearity, with initial data in the energy space
\[
\imath\partial_{t}\psi(t,x)+\triangle\psi+|\psi|^{2\sigma}\psi=0,\qquad \psi(0,x)=\psi_{0}(x)\in\,H^{1}(\mathbb{R}^{d}).
\]
In \cite{CL-CMP82} , using a variational characterization, it was established the orbital stability of the ground-states in the subcritical case, i.e., for  $\sigma<2\backslash d$. For the general case, using  results contained in \cite{SS-CMP85,GSS-87,W-CMP86}
 it is possible to generalize the result on the stability of the ground state solitary waves, extending  the Vakhitov-Kolokolov criterion from spectral stability to the orbital stability. The result provides an alternative proof of  orbital stability for the subcritical solitary waves and shows the orbital instability in the critical and supercritical case ($\sigma d\geqslant 2$).

It turns out that there is a strict relation between blow-up and orbital stability of standing waves \cite{St-PD91}. The NLSE admits blow-up  solutions if and only if its solitary waves are orbitally unstable. This behavior has some relevant exceptions as in the case of NLSE in bounded domains or in \cite{SF-NON08} where the key feature of all this models is always the absence of translational invariance in space.

The analysis of stationary states stability for concentrated nonlinearities traces back to \cite{AN-CMP13,ANO-JMP13,ANO-DCDS16}. For the concentrated NLSE in dimension $2$ the scenario is different and, in some sense, surprising. As it will be illustrated in Section
\ref{focusing} there are, at any fixed value of the mass, two branches of stationary states, distinguished by the value of the frequency $\omega$, with opposite orbital stability behavior. To the best of our knowledge, there is no similar behavior  for a standard Schr\"odinger  equation on $\mathbb{R}^{d}$, but some analogy exists with the 1d NLSE in the presence of a point interaction \cite{FOO-05} and with NLSE on compact domains \cite{NTV-19,PV-17}. In all these cases, the nonlinearity is supercritical.

As for the case of concentrated NLSE in the defocusing case, the scenario is really puzzling. In Section \ref{defocusing} the  analysis of stationary waves reveals that they are stable and moreover that they are ground states.
 
%%%%%%%%%%%%%%%%%%%%%%%%%%%%%%%%%%%%%%%%%%%%%%%%%%%%%%%%%%%%%%%%%%%%%%%%%%%%%%%%%%

\subsection{Setting and known results}

The problem under investigation can be formally written as
\begin{equation}
 \label{eq-cpform}
 \left\{
 \begin{array}{ll}
  \disp\i\f{\de\psi}{\de t}=\big(-\lap+\be|\psi|^{2\si}\delta_{\mathbf{0}}\big)\psi, & \text{in }\Rp\times\Rd,\\[.4cm]
  \psi(0)=\psi_0, & \text{in }\Rd,
 \end{array}
 \right.
\end{equation}
where $\si>0$, $\beta\in\R$ and $\delta_{\mathbf{0}}$ is a Dirac delta function centered  at the origin of $\Rd$. As extensively explained in \cite{ACCT-18,CCT-ANIHPC19}, the ground state problem in  \eqref{eq-cpform}  can be rigorously formulated in the following weak form:
\begin{equation}
 \label{eq-cpweak}
 \left\{
 \begin{array}{l}
  \disp\i\f{d}{dt}\scal{\chi}{\psi(t)}=\scal{\na\chi_\la}{\na\psi(t)}+\la(\scal{\chi_\la}{\phi_\la(t)}-\scal{\chi}{\psi(t)})+\th\big(|q(t)|\big)q_\chi^*q(t),\quad\forall \chi\in V\\[.4cm]
  \psi(0)=\psi_0,
 \end{array}
 \right.
\end{equation}
where $\scal{\cdot}{\cdot}$ is the usual scalar product  in $L^2(\Rd)$, $\th_\la:\Rp\to\R$ is defined as
\[
 \th_{\lambda}(s):=\f{\log(\sqrt{\la}/2)+\ga}{2\pi}+\be s^{2\si}
\]
(with $\ga$ the Euler-Mascheroni constant) and $V$ is the energy space, i.e.
\begin{equation}
 \label{eq-ensp}
 V:=\left\{\chi\in L^2(\Rd):\chi=\chi_\la+q\G_\la,\,\chi_\la\in H^1(\R^2),\,q\in\C\right\},
\end{equation}
with $\la>0$ and $\G_\la$ denoting the Green's function of $-\lap+\la$ in $\Rd$, i.e., 
\begin{equation}
 \label{eq-green}
 \G_\la(\x):=\f{K_0(\sqrt{\la}\x)}{2\pi}=\f{1}{2\pi}\F^{-1}[(|\k|^2+\la)^{-1}](\x),
\end{equation}
(recall that $K_0$ is the Macdonald function of order zero given, e.g., in \cite{AS-65} and $\F$ is the unitary Fourier transform of $\Rd$). Note also that, the parameter $\la$ does not affect the definition of $V$ nor the Cauchy problem \eqref{eq-cpweak}. Indeed, it is possible to rewrite the space $V$ without the  parameter $\lambda$, as
\begin{equation}\label{vi}
V=\left\{ \chi\in L^{2}(\mathbb{R}^{2}),\, \chi=\chi_{0}-q\frac{\log |x|}{2\pi},\,\chi_{0}\in \dot{H}^{1}(\mathbb{R}^{2}), \,\, q\in \mathbb{C}\right\}
\end{equation}
where $\dot{H}^{1}(\mathbb{R}^{2})$ is the homogeneous Sobolev space. It is important to remark however that the parameter $q$ appearing in the decomposition above does not coincide with the analogous parameter in (\ref{eq-ensp}). Furthermore, (\ref{vi}) is not easily implemented in the expression of the energy, so that  we shall keep using (\ref{eq-ensp}). Coherently with this choice, we shall decompose the solution $\psi(t)$ as 
\begin{equation}\label{decomp}
\psi(t)=\phi_{\lambda}(t)+q(t)\,G_{\lambda},\quad\phi_{\lambda}(t)\in H^{1}(\mathbb{R}), q(t)\in\mathbb{C}\end{equation}
and refer to $\phi_{\lambda}(t)$, $q(t)G_{\lambda }$ and $q(t)$ as to the regular part, the singular part and the charge, respectively.
The decomposition \eqref{decomp} makes sense, as it has been proven in \cite{ACCT-18,CCT-ANIHPC19} that, for $\si\geq1/2$, \eqref{eq-cpweak} is locally well-posed in $V$ (with the additional assumption $\phi_{\lambda}(0)\in H^{1+\eta}$, $\eta>0$) and that the mass
\[
 \M(t)=\M\big(\psi(t)\big):=\|\psi(t)\|^2,
\]
$\|\cdot\|$ denoting the usual norm in $L^2(\Rd)$, and the energy
\begin{multline}
 \label{eq-energy}
 \E(t)=\E\big(\psi(t)\big):=\\
 =\|\na\phi_\la(t)\|^2+\la\big(\|\phi_\la(t)\|^2-\|\psi(t)\|^2\big)+\bigg(\f{\beta|q(t)|^{2\sigma}}{\sigma+1}+\f{\log(\sqrt{\la}/2)+\ga}{2\pi}\bigg)|q(t)|^2,
\end{multline}
which is independent of $\la$ as well, are preserved along the flow. In addition, when $\be>0$, i.e., in the defocusing case, the solution is global in time, whereas when $\be<0$, i.e., in the focusing case, the solution blows up in a finite time. In order to prove these  results, one has to require \cite{ACCT-18} that  $\phi_{\lambda}(0)$  belongs to the Schwartz space, which is only a technical hypothesis, and, more important, its energy satisfies
\[
 \E(\psi_0)<\La=\La(\si,\be):=-\f{\si}{4\pi(\si+1)(-4\pi\si\be)^{1/\si}}.
\]
In the following sections we study the problem of the stability of stationary states separately in the focusing and defocusing case.

%%%%%%%%%%%%%%%%%%%%%%%%%%%%%%%%%%%%%%%%%%%%%%%%%%%%%%%%%%%%%%%%%%%%%%%%%%%%%%%%%%
%%%%%%%%%%%%%%%%%%%%%%%%%%%%%%%%%%%%%%%%%%%%%%%%%%%%%%%%%%%%%%%%%%%%%%%%%%%%%%%%%%
%%%%%%%%%%%%%%%%%%%%%%%%%%%%%%%%%%%%%%%%%%%%%%%%%%%%%%%%%%%%%%%%%%%%%%%%%%%%%%%%%%

\section{Focusing case}\label{focusing}

In  the focusing case, i.e., for  $\be<0$, \eqref{eq-cpweak} admits (see \cite{ACCT-18}) a unique family of standing waves of the form
\begin{equation}
 \label{eq-stan1}
 \psi_\om(t,\x):=\exp^{\i\om t}\,\exp^{\i\eta}\,u_\om(\x),\qquad\eta\in\R,\quad \om\in(\wt{\om},+\infty), \quad  \wt{\om}:=4\exp^{-2\ga},
\end{equation}
where
\begin{equation}
 \label{eq-stan2}
 u_\om(\x):=q(\om)\G_\om(\x),\qquad q(\om):=\bigg(-\f{\log\big(\sqrt{\om}/2\big)+\ga}{2\pi\be}\bigg)^{1/2\si}.
\end{equation}
The behavior of $q(\om)$ is depicted in Figure \ref{fig-val1}(a).

Now, plugging \eqref{eq-stan2} into \eqref{eq-energy}, one finds that the energy of the standing waves as  a function of  the  frequency $\om$  reads
\begin{equation}
 \label{eq-Eomega}
 \E(\omega):=\E(u_\om)=\bigg(\f{\sigma\log\big(\sqrt\omega/2 \big)+\gamma\sigma}{2\pi(\sigma+1)}-\frac{1}{4\pi}\bigg)\bigg(-\f{\log\big(\sqrt{\om}/2\big)+\ga}{2\pi\be}\bigg)^{1/\si},\qquad\forall \om\in(\wt{\om},+\infty).
\end{equation}
The behavior of $E(\om)$ is represented in Figure \ref{fig-val1}(b). In addition,
\begin{equation}
 \label{eq-ommin}
 \min_{\om\in(\wt{\om},+\infty)}\E(\om)=\E(\ov{\om})=\La,\qquad\text{where}\quad \ov{\om}:={4\exp^{-2\ga+1/\si}}.
\end{equation}

\begin{figure}[ht!]
\centering
\subfloat[][Behavior of $q(\om)$.]
{\includegraphics[width=0.40\textwidth]{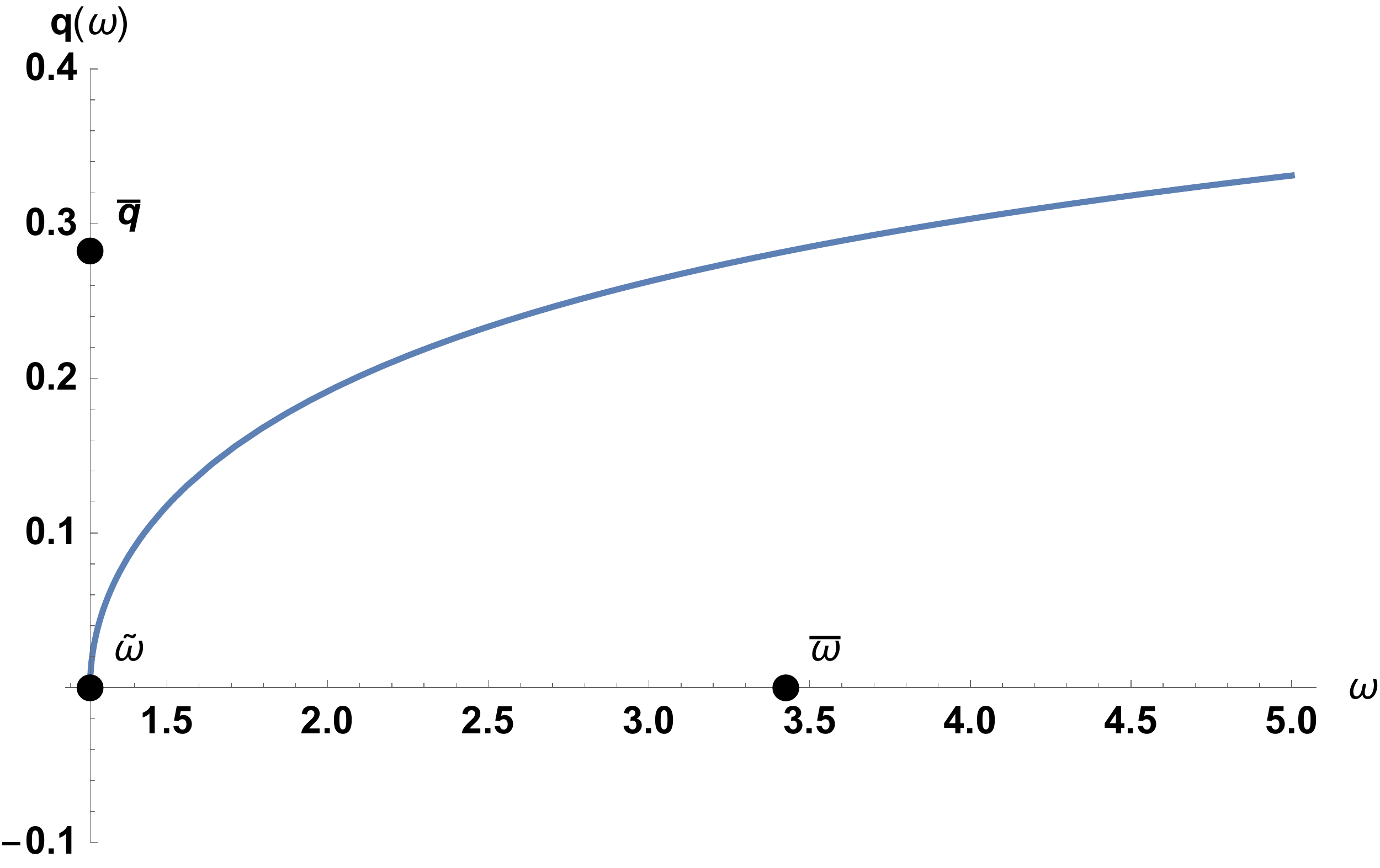}}
\hspace{1cm}
\subfloat[][Behavior of $\E(\om)$.]
{\includegraphics[width=0.40\textwidth]{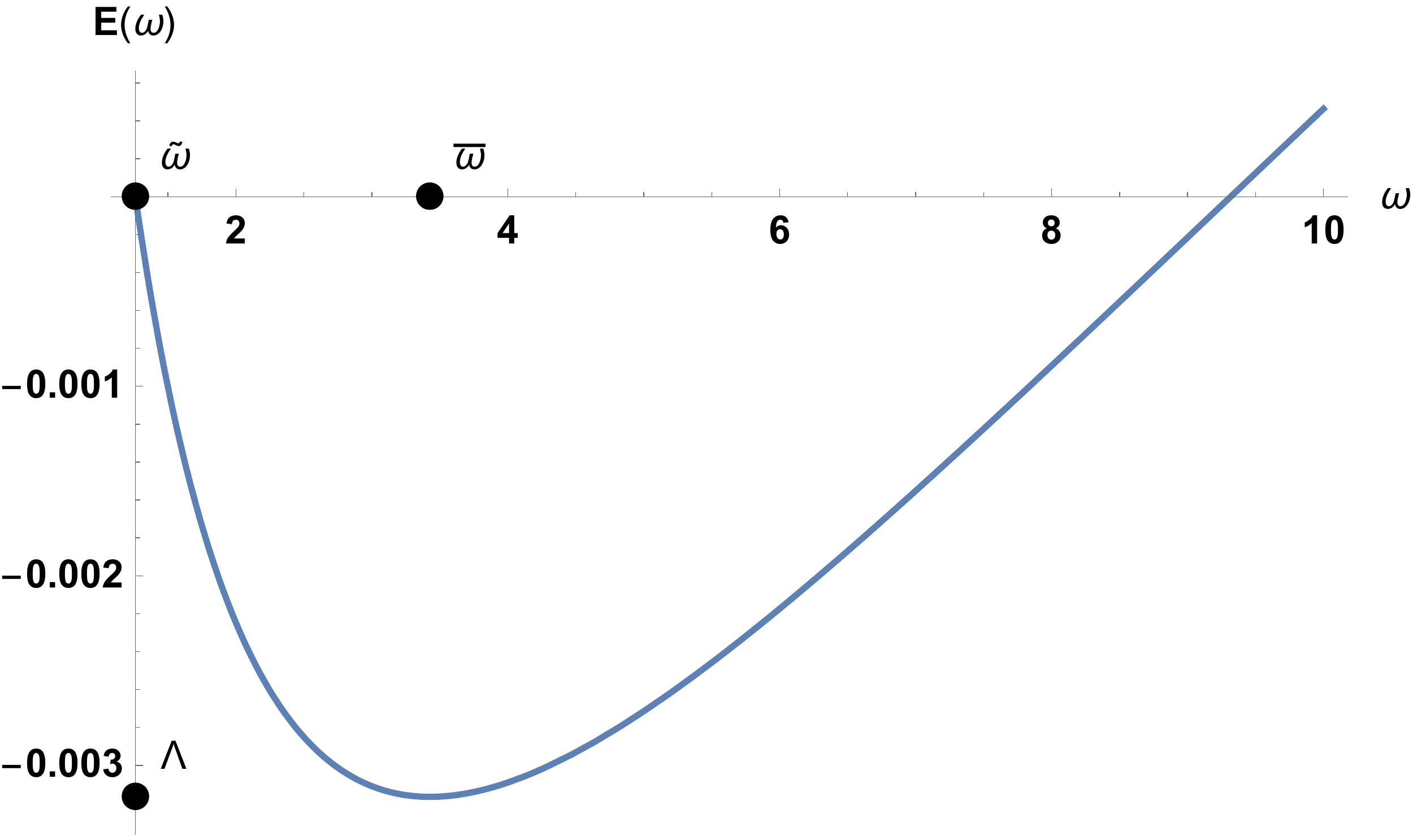}}
\caption{Plots of $q(\om)$ and $\E(\om)$ for $ \omega \in (\wt{\om},+\infty)$,  when $\si=1$ and $\be=-1$. Here $\wt{\om}\approx1.26$, $\ov{\om}\approx3.43$, $\ov{q}\approx0.2$ and $\Lambda\approx-0.0016$.}
\label{fig-val1}
\end{figure}

On the other hand, noting that $q(\wt{\om})=0$ and that $q(\cdot)$ is smooth and strictly increasing on $(\wt{\om},+\infty)$, one can  take the inverse $q(\omega)$ of the function

\begin{equation}
 \label{eq-omdiq}
 \omega(q):=4\,\exp^{-2\ga-4\pi\be q^{2\si}},\qquad q>0,
\end{equation}
and plug it into \eqref{eq-Eomega}, to obtain the energy as a function of $q$, i.e.,
\begin{equation}
 \label{eq-Eq}
 \E(q)={-\f{q^2}{4\pi}-\f{\si\be q^{2\si+2}}{\si+1}}.
\end{equation}
The behaviors of $\omega(q)$ and $\E(q)$ are depicted in Figure \ref{fig-val2}(a) and \ref{fig-val2}(a), respectively. This alternative form can be useful in computation since \eqref{eq-Eq} is  more manageable than \eqref{eq-Eomega}. Furthermore,
\[
 \inf_{q>0}\E(q)=\E(\ov{q})=\La < 0,\qquad\text{where}\quad \ov{q}:=q(\ov{\omega})=(-4\pi\si\be)^{-1/2\si}.
\]

\begin{figure}[ht]
\centering
\subfloat[][Behavior of $\log(\om(q))$.]
{\includegraphics[width=0.45\textwidth]{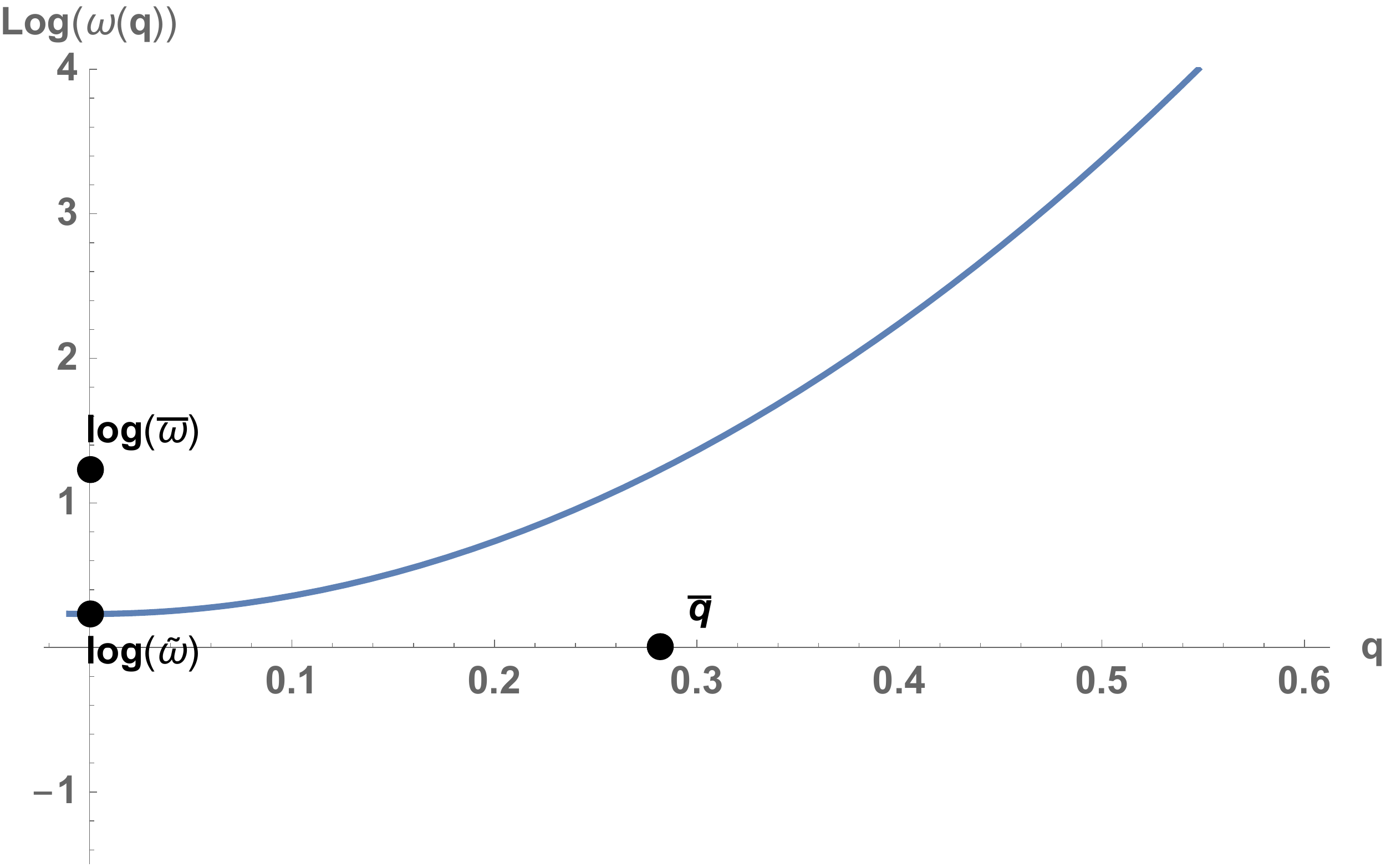}}
\hspace{1cm}
\subfloat[][Behavior of $\E(q)$.]
{\includegraphics[width=0.45\textwidth]{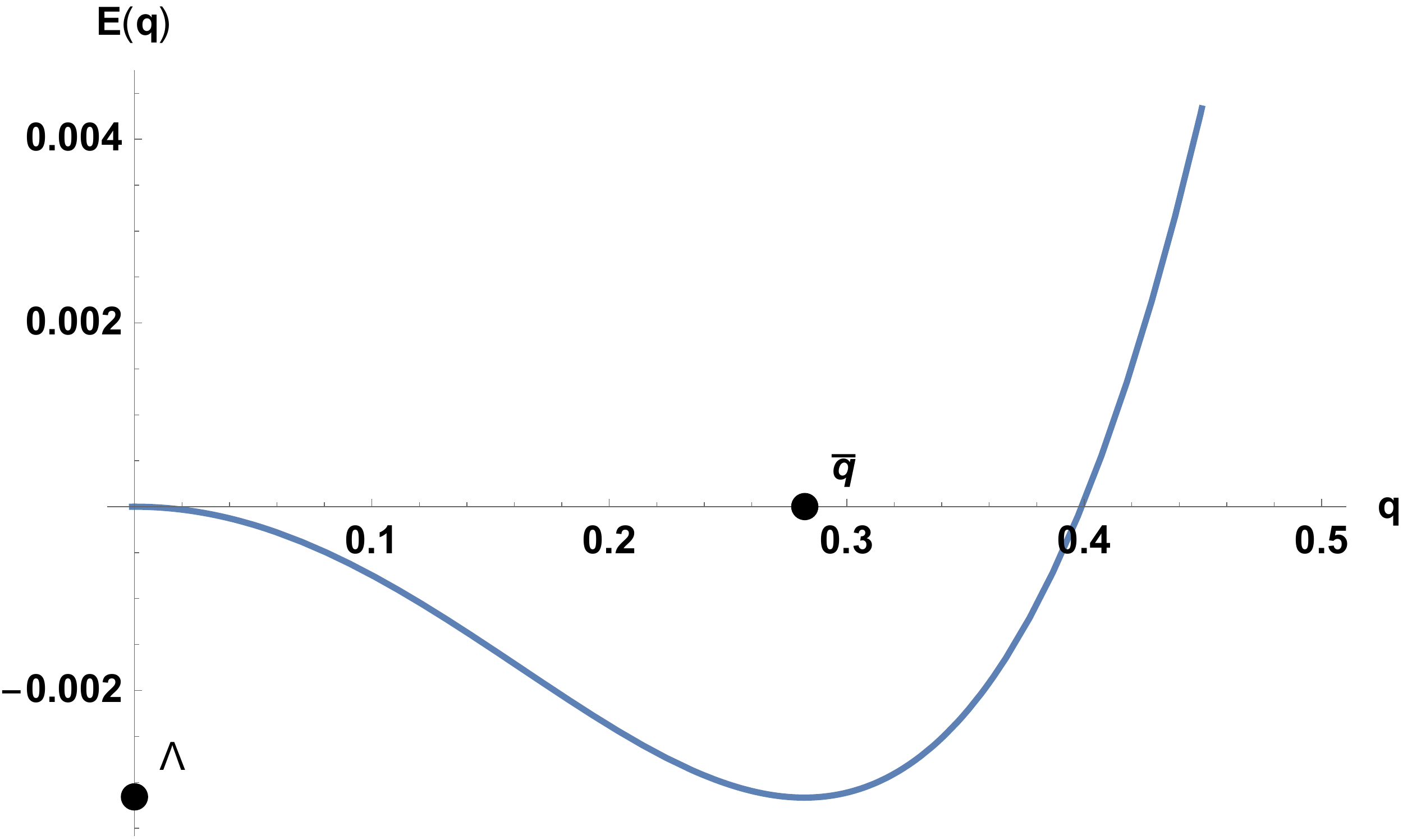}}
\caption{Plots of $\om(q)$ and $\E(q)$ for $q \in \Rp$,  when $\si=1$ and $\be=-1$.}
\label{fig-val2}
\end{figure}

The natural question arising at this point is about the stability of the standing waves. In view of the application of Grillakis-Shatah-Strauss theory \cite{GSS-87}, it is first necessary to compute the mass $\M$ as a function of $\omega$ and $q$.
Exploiting \eqref{eq-green}, one has that
\begin{equation}
 \label{eq-massom}
 \M(\om):=\M(u_\om)=\f{q^2(\om)}{4\pi\om}=\f{1}{4\pi\om}\bigg(-\f{\log(\sqrt{\om}/2)+\gamma}{2\pi\be}\bigg)^{1/\si}.
\end{equation}
On the other hand, one can easily check that
\[
 \M'(\omega)=\f{q^2(\om)}{4\pi\om^2}\underbrace{\bigg[\f{\left(\log(\sqrt{\om}/2)+\gamma\right)^{-1}}{2\si}-1\bigg]}_{=:h(\om)},
\]
whence
\begin{align}
 \label{eq-mass-mon}
 \M'(\om)>0\quad\text{(resp. $\M'(\om)<0$)}\qquad & \Longleftrightarrow \qquad h(\om)>0 \quad\text{(resp. $h(\om)<0$)}\nonumber\\[.2cm]
 \qquad & \Longleftrightarrow \qquad \wt{\om}<\om<\ov{\om}\quad\text{(resp. $\om>\ov{\om}$)}.
\end{align}
In addition, as $\lim_{\om\to\wt{\om}}\M(\om)=\lim_{\om\to+\infty}\M(\om)=0$, there results
\[
 \sup_{\om\in(\wt{\om},+\infty)}\M(\om)=\M(\ov{\om})=\f{\exp^{2\gamma-1/\si}}{16\pi(-4\pi\si\be)^{1/\si}}=:\ov{\mu}.
\]
As a consequence, for every value of the mass $\mu\in(0,\ov{\mu})$ (or, alternatively, of the energy $\E\in(\La,0)$) there exists two distinct families of standing waves $u_{\om_1},\,u_{\om_2}$, such that $\M(\om_1)=\M(\om_2)=\mu$, with $\om_1\in(\wt{\om},\ov{\om})$ and $\om_2\in(\ov{\om},+\infty)$.

Analogous results can be obtained writing the mass of the standing waves in terms of $q$ in place of $\omega$, so that
\begin{equation}
 \label{eq-massq}
 \M(q)=\f{q^2\exp^{2\ga+4\pi\be q^{2\si}}}{16\pi}
\end{equation}
and
\[
 \sup_{q>0}\M(q)=\M(\ov{q})=\ov{\mu}.
\]
The qualitative behavior of $\M(\om)$ and $\M(q)$ is depicted in Figure \ref{fig-val3}.

\begin{figure}[ht]
\centering
\subfloat[][Behavior of $\M(\omega)$.]
{\includegraphics[width=0.45\textwidth]{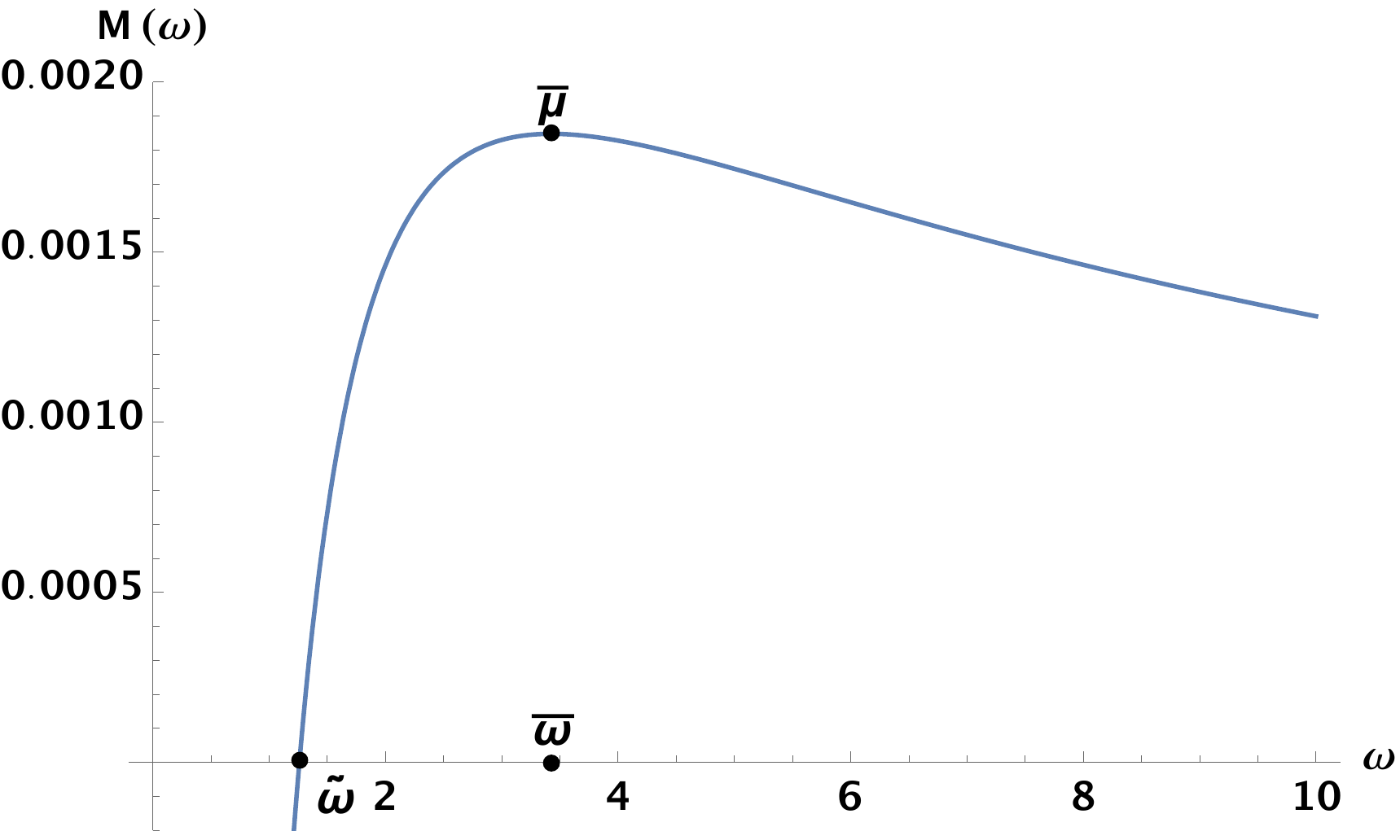}}
\hspace{1cm}
\subfloat[][Behavior of $\M(q)$.]
{\includegraphics[width=0.45\textwidth]{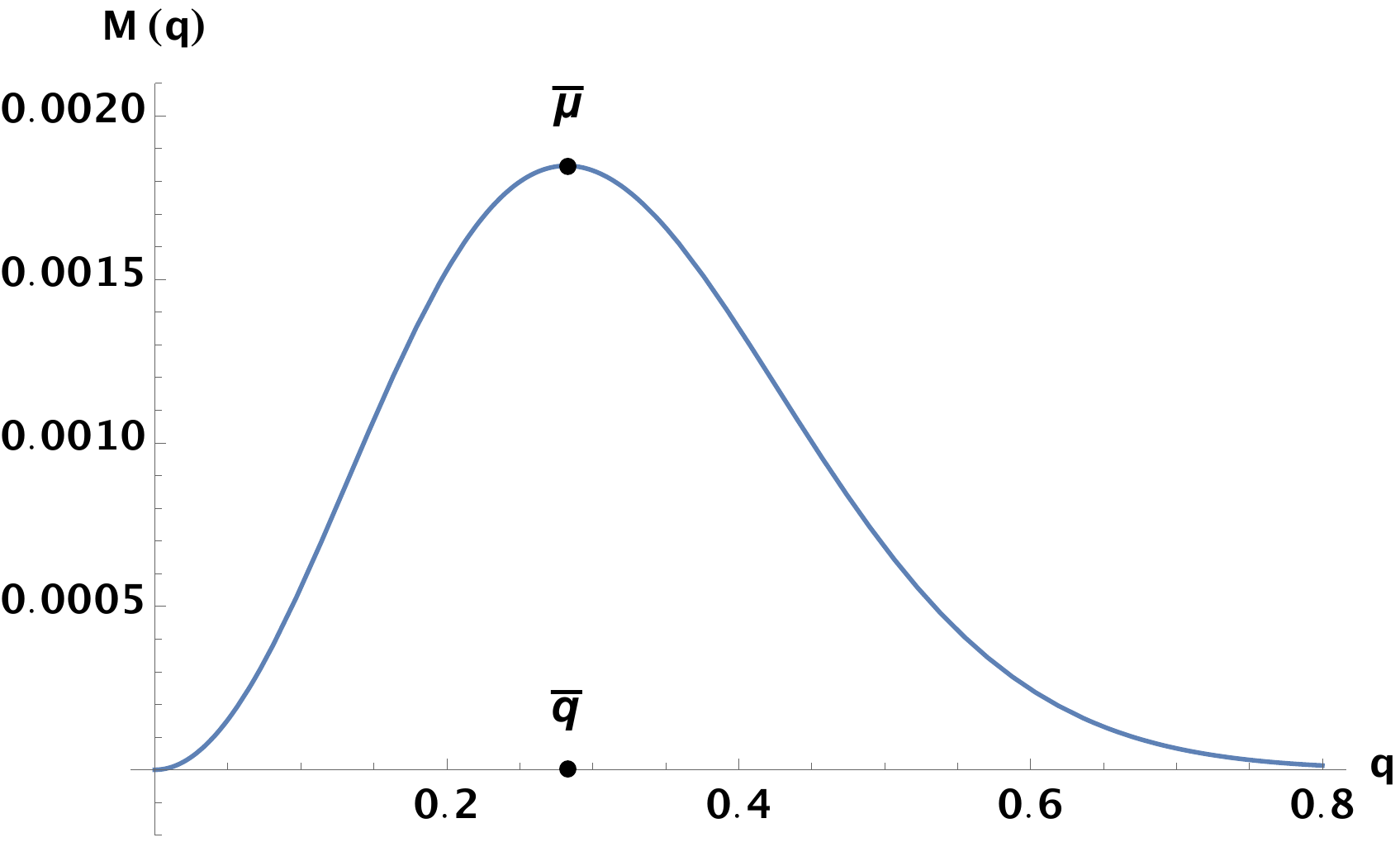}}
\caption{Plots of $\M(\omega)$ and $\M(q)$ when $\si=1$ and $\be=-1$.}
\label{fig-val3}
\end{figure}

For any $\mu>0$ there is no ground state of mass $\mu$, i.e., no global minimizer of the energy constrained on 
\[
 V_\mu:=\{\psi\in V:\M(\psi)=\mu\}.
\]
This can be easily seen if one defines a sequence $\{u_n\}_{n\in\mathbb{N}}$ such that
\[
u_{n}(x):=2\sqrt{\pi\mu\,n}\,\G_1(\sqrt{n}\,x), \qquad \M(u_{n})=\mu.
\]
Indeed, $\left\{u_n\right\}\subset V_\mu$ and
\[
 \E(u_n)=-\sqrt{n}\mu+\bigg(\f{\be (4\pi\mu\,n)^{\sigma}}{\si+1}+\f{\log(\sqrt{n}/2)+\ga}{2\pi}\bigg)(\pi\mu\,n)^{1/ 2} \xrightarrow[n \to + \infty]{}-\infty,
\]
since $\be<0$. Hence, the stability analysis requires the use of the techniques developed in \cite{GSS-87}, as shown in  Theorem \ref{thm-sfoc}. 

First, we recall the definition of \emph{orbitally} stable standing wave. To this aim, preliminarily, we endow the energy space $V$ with a norm. Due to the several possible decompositions of a function $\chi\in V$ for different values of the spectral parameter $\la>0$ (see \eqref{eq-ensp}), in order to obtain a suitable norm, one has to fix a value $\la=\ov{\la}>0$ and 
then set
\[
 \|\chi\|_{\ov{\la}}^2:=\|\chi_{\ov{\la}}\|_{H^1(\R^2)}^2+\f{|q|^2}{4\pi\ov{\la}}.
\]
Clearly, any other choice of $\la$ gives rise to an equivalent norm. In this section we will set $\ov{\la}=1$ for the sake of simplicity.

\begin{defn}
 The standing wave $u_\om$ is said to be \emph{orbitally stable} whenever for every $\varepsilon>0$ there exists $\delta>0$ such that: if $\|\psi_0-e^{\imath\,\eta}u_\om\|_1<\delta$, for some $\eta\in\R$, and $\psi(t)$ is a solution of \eqref{eq-cpweak} on $[0,T^*)$ with initial condition $\psi_0$, then $\psi(t)$ can be continued to a solution on $[0,+\infty)$ and
 \[
  \sup_{t\in\Rp}\,\inf_{\eta\in\R}\left\|\psi(t)-\exp^{\imath\eta}u_\omega\right\|_1<\varepsilon.
 \]
 Otherwise the standing wave is called \emph{unstable}.
\end{defn}

\begin{thm}[Stability in the focusing case]
 \label{thm-sfoc}
 Let $\sigma\geq1/2$ and $\be<0$. The standing waves defined in \eqref{eq-stan1}-\eqref{eq-stan2} are orbitally stable if $\om\in(\wt{\om},\ov{\om})$ and unstable if $\omega>\ov{\om}$ (where $\ov{\om}$ is given in \eqref{eq-ommin}).
\end{thm}

Note that the previous theorem entails that, for every mass $\mu\in(0,\ov{\mu})$, there is a pair of standing waves of mass $\mu$, where the one with low frequency is stable, while the one with high frequency is unstable.

\begin{rem}
The assumption $\sigma\geq1/2$ is only related to the local well-posedness of \eqref{eq-cpweak} proved in \cite{CCT-ANIHPC19}. It is likely that it could be dropped by means of a more refined analysis of the local well-posedness and hence is not actually relevant in the stability analysis.
\end{rem}

\begin{proof}[Proof of Theorem \ref{thm-sfoc}]
 The proof is based on \cite[Theorem 3]{GSS-87}. Assumptions 1 of \cite[Theorem 3]{GSS-87} is clearly satisfied as \eqref{eq-cpweak} is locally well-posed with preserved mass and energy, while the fulfillment of Assumption {2} is a direct consequence of the form of the standing waves defined in \eqref{eq-stan1}-\eqref{eq-stan2}.

Concerning Assumption 3, we define the action functional associated with \eqref{eq-cpweak}, namely
 \[
  \S_\om:V\to\R,\qquad\S_\om(u):=\E(u)+\om\M(u).
 \]
Recall that, as $u_\om$ is a standing wave, $d\S_\om(u_\om)=0$, where $d\S_\om$ denotes the Fr\'echet differential. Then, define  the operator
 \begin{equation}
  \label{eq-Hom}
  H_\om:V\to V^*,\qquad H_\om:=d^2\S_\om(u_\om).
 \end{equation}
We have to prove that for every $\om\in(\wt{\om},+\infty)$:
 \begin{itemize}
  \item[(i)] $H_\om$ has exactly one negative simple eigenvalue;
  \item[(ii)] the kernel of $H_\om$ coincides with the span of $u_\om$;
  \item[(iii)] the rest of $\sigma(H_\om)$ is positive and bounded away from zero.
 \end{itemize}

 As $d^2\M(u_\om)=2\times\Id$, it is necessary to compute only $d^2\E(u_\om)$. Since $\E$ is a functional of class $C^2$ we can compute the G$\hat{\mathrm{a}}$teaux second differential in place of the Fr\'echet second differential, i.e.,
 \[
  d^2\E(u_\om)[h,k]=\left.\f{\partial^2 \E(u_\om+\nu h+\tau k)}{\partial\nu\partial\tau}\right|_{\nu=\tau=0}.
 \]
 In addition, for the sake of simplicity, we can set $\lambda=\om$ in the definition of $\E$. Therefore,  standard computations  yields \begin{multline*}
  \f{\partial \E(u_\om+\nu h+\tau k)}{\partial\tau}= 2\Re\bigg\{\scal{\nu\nabla h_\om+\tau\nabla k_\om}{\nabla k_\om}+\om\big(\scal{\nu h_\om+\tau k_\om}{k_\om}-\scal{u_\om+\nu h+\tau k}{k}\big)+\\[.2cm]
  +q_k\big(q^*(\om)+\nu q_h^*+\tau q_k^*\big)\f{\log(\sqrt{\om}/2)+\ga}{2\pi}+\be q_k\big(q(\om)+\nu q_h+\tau q_k\big)^{\si}\big(q^*(\om)+\nu q_h^*+\tau q_k^*\big)^{\si+1}\bigg\},
 \end{multline*}
 so that
 \begin{multline*}
  \f{\partial^2 \E(u_\om+\nu h+\tau k)}{\partial\nu\partial\tau}= 2\Re\bigg\{\scal{\nabla h_\om}{\nabla k_\om}+\om\big(\scal{h_\om}{k_\om}-\scal{h}{k}\big)+\\[.2cm]
  +q_kq_h^*\f{\log(\sqrt{\om}/2)+\ga}{2\pi}+\si\be q_kq_h\big(q(\om)+\nu q_h+\tau q_k\big)^{\si-1}\big(q^*(\om)+\nu q_h^*+\tau q_k^*\big)^{\si+1}+\\[.2cm]
  +(\si+1)\be q_kq_h^*\big(q(\om)+\nu q_h+\tau q_k\big)^{\si}\big(q^*(\om)+\nu q_h^*+\tau q_k^*\big)^{\si}\bigg\}
 \end{multline*}
 and hence
 \begin{multline}
  \label{eq-dueE}
  \left.\f{\partial^2 \E(u_\om+\nu h+\tau k)}{\partial\nu\partial\tau}\right|_{\nu=\tau=0}= 2\Re\big\{\scal{\nabla h_\om}{\nabla k_\om}+\om\big(\scal{h_\om}{k_\om}-\scal{h}{k}\big)\big\}+\\[.2cm]
  +\f{\log(\sqrt{\om}/2)+\ga}{\pi}\Re\{q_kq_h^*\}+2\be q^{2\si}(\om)\Re\{\si q_k^*q_h^*+(\si+1)q_k^*q_h\}.
 \end{multline}
 Now, if we split each quantity as real and imaginary part, i.e.,
 \begin{gather*}
  h=h^r+\imath h^i, \quad k=k^r+\imath k^i,\\[.2cm]
  h_\om=h_\om^r+\imath h_\om^i, \quad k_\om=k_\om^r+\imath k_\om^i,\\[.2cm]
  q_h=q_h^r+\imath q_h^i, \quad q_k=q_k^r+\imath q_k^i,
 \end{gather*}
 then \eqref{eq-dueE} reads
 \[
  \left.\f{\partial^2 \E(u_\om+\nu h+\tau k)}{\partial\nu\partial\tau}\right|_{\nu=\tau=0}=B_1[h^r,k^r]+B_1[h^i,k^i],
 \]
 where $B_1,\,B_2$ are two sesquilinear forms given by
 \begin{multline*}
  B_1[h^r,k^r]:=2\big(\scal{\nabla h_\om^r}{\nabla k_\om^r}+\om\big(\scal{h_\om^r}{k_\om^r}-\scal{h^r}{k^r}\big)\big)\\[.2cm]
  +\left(\f{\log(\sqrt{\om}/2)+\ga}{\pi}+2\be (2\si+1)q^{2\si}(\om)\right)q_k^rq_h^r
 \end{multline*}
 and
 \[
  B_2[h^i,k^i]:=2\big(\scal{\nabla h_\om^i}{\nabla k_\om^i}+\om\big(\scal{h_\om^i}{k_\om^i}-\scal{h^i}{k^i}\big)\big)
  +\left(\f{\log(\sqrt{\om}/2)+\ga}{\pi}+2\be q^{2\si}(\om)\right)q_k^iq_h^i.
 \]
 Furthermore, one notes that $B_1,\,B_2$ are the sesquilinear form (restricted to real-valued functions) associated with the operators $H_{\al_1},\,H_{\al_2}:L^2(\R^2)\to L^2(\R^2)$ with domains
 \begin{multline}
  \label{eq-opdelta1}
 \mathrm{dom}(H_{\al_j}):=\bigg\{\psi\in L^2(\R^2):\psi=\phi_\la+q\G_\la,\,\phi_\la\in H^2(\R^2),\,q\in \C,\\[.2cm]
 \phi_\la(\z)=\bigg(\al_i+\f{\log(\sqrt{\la}/2)+\ga}{2\pi}\bigg)q\bigg\},\qquad i=1,\,2,
\end{multline}
with $\la>0$, and action
\begin{equation}
 \label{eq-opdelta2}
 (H_{\al_i}+\la)\psi:=(-\Delta+\la)\phi_\la,\qquad\forall\psi\in\mathrm{dom}(H_{\al_i}),\qquad i=1,\,2,
\end{equation}
where
\begin{equation}
 \label{eq-al12}
 \al_1=(2\si+1) \beta q^{2\si}(\om),\qquad \al_2=\beta q^{2\si}(\om).
\end{equation}
Summing up, 
\[
 \left.\f{\partial^2 \E(u_\om+\nu h+\tau k)}{\partial\nu\partial\tau}\right|_{\nu=\tau=0}=2(h^r,h^i)\begin{pmatrix} H_{\al_1} & 0 \\ 0 & H_{\al_2} \end{pmatrix}\begin{pmatrix} k^r \\ k^i \end{pmatrix},
\]
whence
\[
 d^2\E(u_\om)=2\begin{pmatrix} H_{\al_1} & 0 \\ 0 & H_{\al_2} \end{pmatrix}
\]
and, consequently,
\[
 H_\om=2\begin{pmatrix} H_{\al_1}+\om & 0 \\ 0 & H_{\al_2}+\om \end{pmatrix}.
\]
In order to verify properties (i), (ii) and (iii), it suffices to observe that
\[
 \sigma(H_{\al_1})=\left\{-\om\,\exp^{-8\pi\,\be\,\si q^{2\si}}\right\}\cup[0,+\infty),\qquad \sigma(H_{\al_2})=\{-\om\}\cup[0,+\infty),
\]
with $-\om\,\exp^{-8\pi\be\,\si q^{2\si}}$ and $-\om$ simple eigenvalues, and that $u_\om$ is the eigenfunction associated with $-\om$. Indeed, this entails 
\begin{equation}
 \label{eq-spectrum}
 \sigma(H_\om)=\left\{\om(1-\exp^{-8\pi\be\si q^{2\si}})\right\}\cup\{0\}\cup[\om,+\infty),
\end{equation}
which  proves that $H_\om$ possesses one simple negative eigenvalue (since $1-\exp^{-8\pi\be\si q^{2\si}}<0$, as $\be<0$), that the kernel of $H_\om$ is the span of $u_\om$ and that the rest of the spectrum is positive and bounded away from zero.

Finally, in order to conclude, it is just sufficient to verify for which values of $\om$ the scalar function
 \[
  D:(\wt{\om},+\infty)\to\R,\qquad D(\om):=\S_\om(u_\om),
 \]
 is strictly convex, which implies that $u_\om$ is stable, or strictly concave, which implies that $u_\om$ is unstable. However, by the properties of the standing waves, $D''(\om)=\M'(\om)$ and therefore, recalling \eqref{eq-mass-mon}, one concludes the proof.
\end{proof}

%%%%%%%%%%%%%%%%%%%%%%%%%%%%%%%%%%%%%%%%%%%%%%%%%%%%%%%%%%%%%%%%%%%%%%%%%%%%%%%%%%
%%%%%%%%%%%%%%%%%%%%%%%%%%%%%%%%%%%%%%%%%%%%%%%%%%%%%%%%%%%%%%%%%%%%%%%%%%%%%%%%%%
%%%%%%%%%%%%%%%%%%%%%%%%%%%%%%%%%%%%%%%%%%%%%%%%%%%%%%%%%%%%%%%%%%%%%%%%%%%%%%%%%%

\section{Defocusing case}\label{defocusing}

One can easily check that there exists a family of standing waves  in the defocusing case  $\be>0$ as well:
\begin{equation}
 \label{eq-stan3}
 \psi_\om(t,\x):=\exp^{\i\om t}\,\exp^{\i\eta}\,u_\om(\x),\qquad u_\om(\x):=q(\om)\G_\om(\x),\qquad q(\om):=\bigg(-\f{\log\big(\sqrt{\om}/2\big)+\ga}{2\pi\be}\bigg)^{1/2\si}
\end{equation}
with $\eta\in\R$, defined for
\[
 \om\in(0,\wt{\om}),\qquad\text{where}\quad \wt{\om}:=4\exp^{-2\ga}.
\]
The behavior of $q(\om)$ is shown in Figure \ref{fig-val4}(a).

In addition, simple computations show that the form of $\E(\om)$ is still given by \eqref{eq-Eomega}, but in this case the function $\E(\om)$ is unbounded from below, due to the fact that $\beta>0$. The behavior of $E(\om)$ is depicted in Figure \ref{fig-val4}(b).

\begin{figure}[ht]
\centering
\subfloat[][Behavior of $q(\om)$.]
{\includegraphics[width=0.45\textwidth]{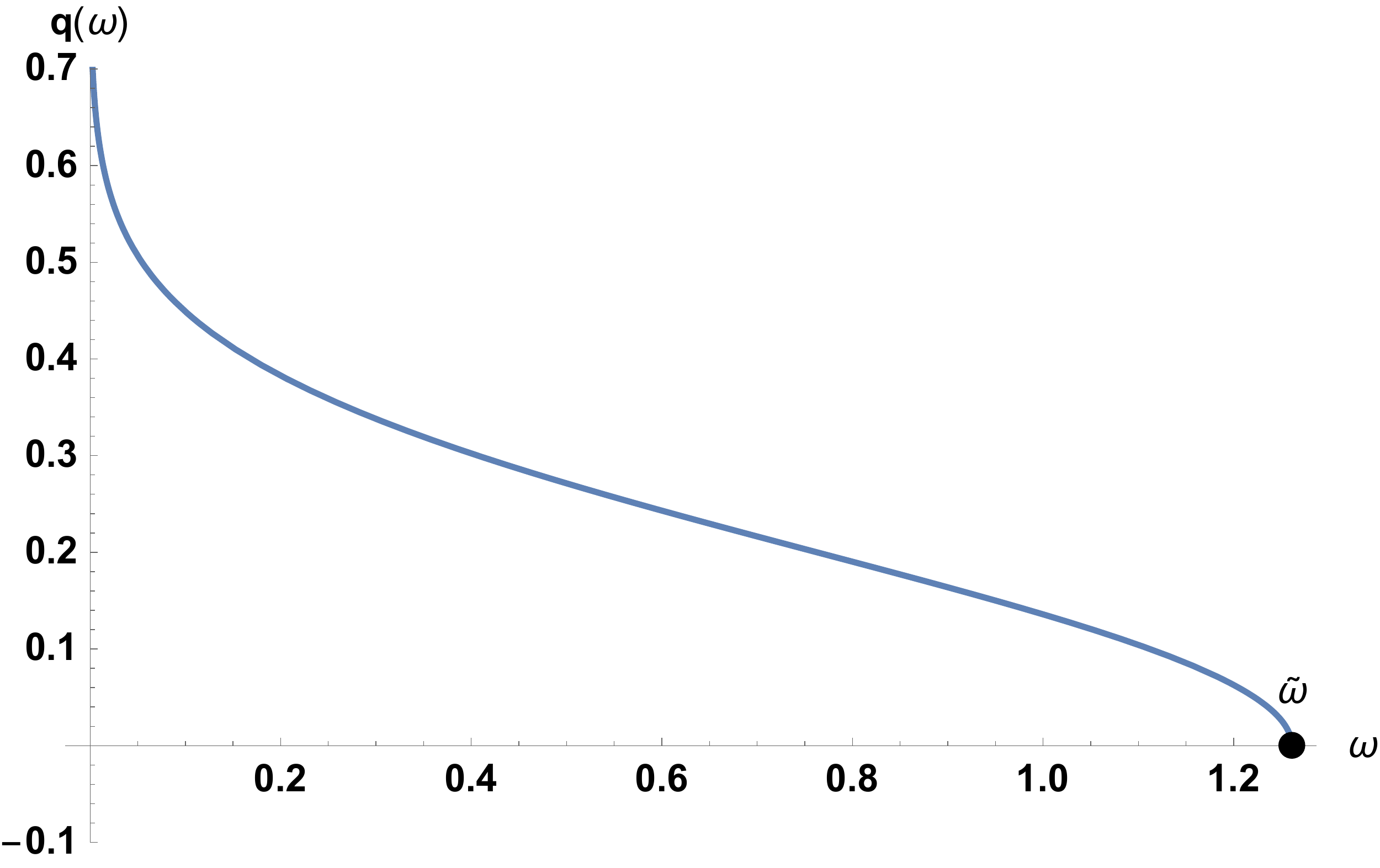}}
\hspace{1cm}
\subfloat[][Behavior of $\E(\om)$.]
{\includegraphics[width=0.45\textwidth]{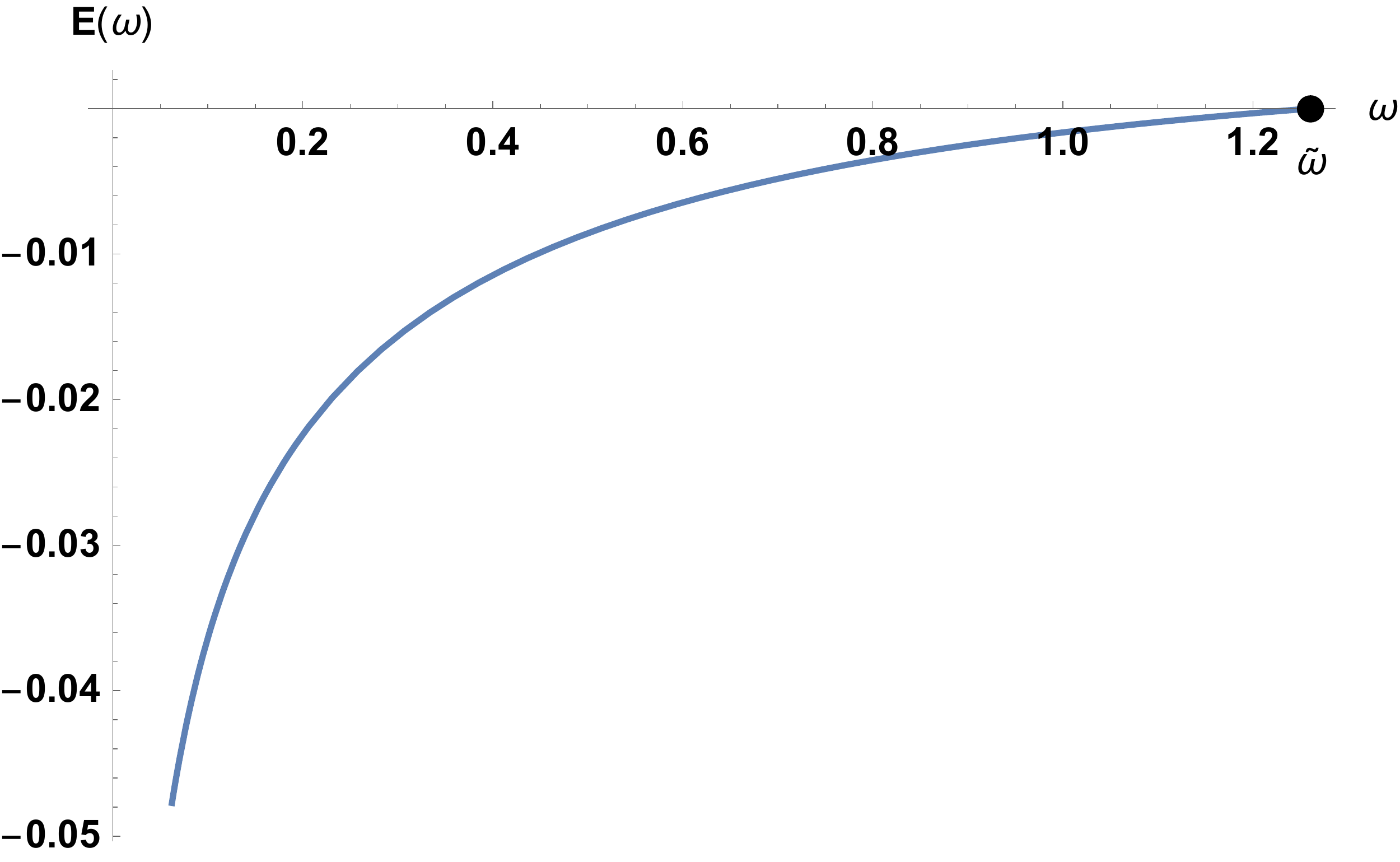}}
\caption{Plots of $q(\om)$ and $\E(\om)$ for $ \omega \in (0,\wt{\om})$,  when $\si=1$ and $\be=1$.}
\label{fig-val4}
\end{figure}

From (\ref{eq-stan3}) one has  that the function $q(\omega)$ is invertible. Again we get that  $\omega(q)$ reads as \eqref{eq-omdiq}
and, plugging \eqref{eq-omdiq} into \eqref{eq-Eomega}, one obtains \eqref{eq-Eq} for $E(q)$. 
The behavior of $\omega(q)$ and $\E(q)$ is  depicted in Figure \ref{fig-val5}(a) and \ref{fig-val5}(b). 

\begin{figure}[ht!]
\centering
\subfloat[][Behavior of $\om(q)$.]
{\includegraphics[width=0.40\textwidth]{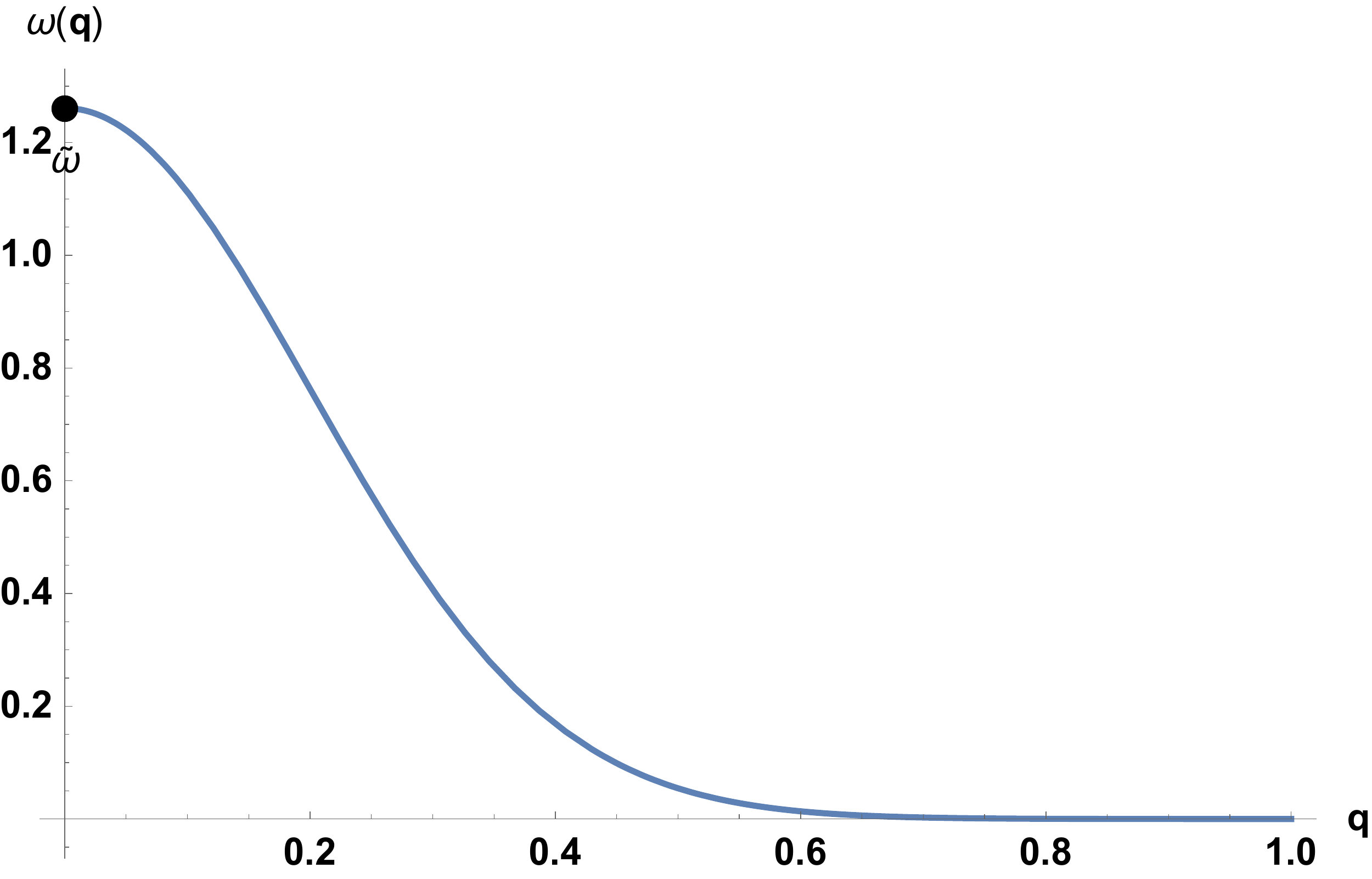}}
\hspace{1cm}
\subfloat[][Behavior of $\E(q)$.]
{\includegraphics[width=0.40\textwidth]{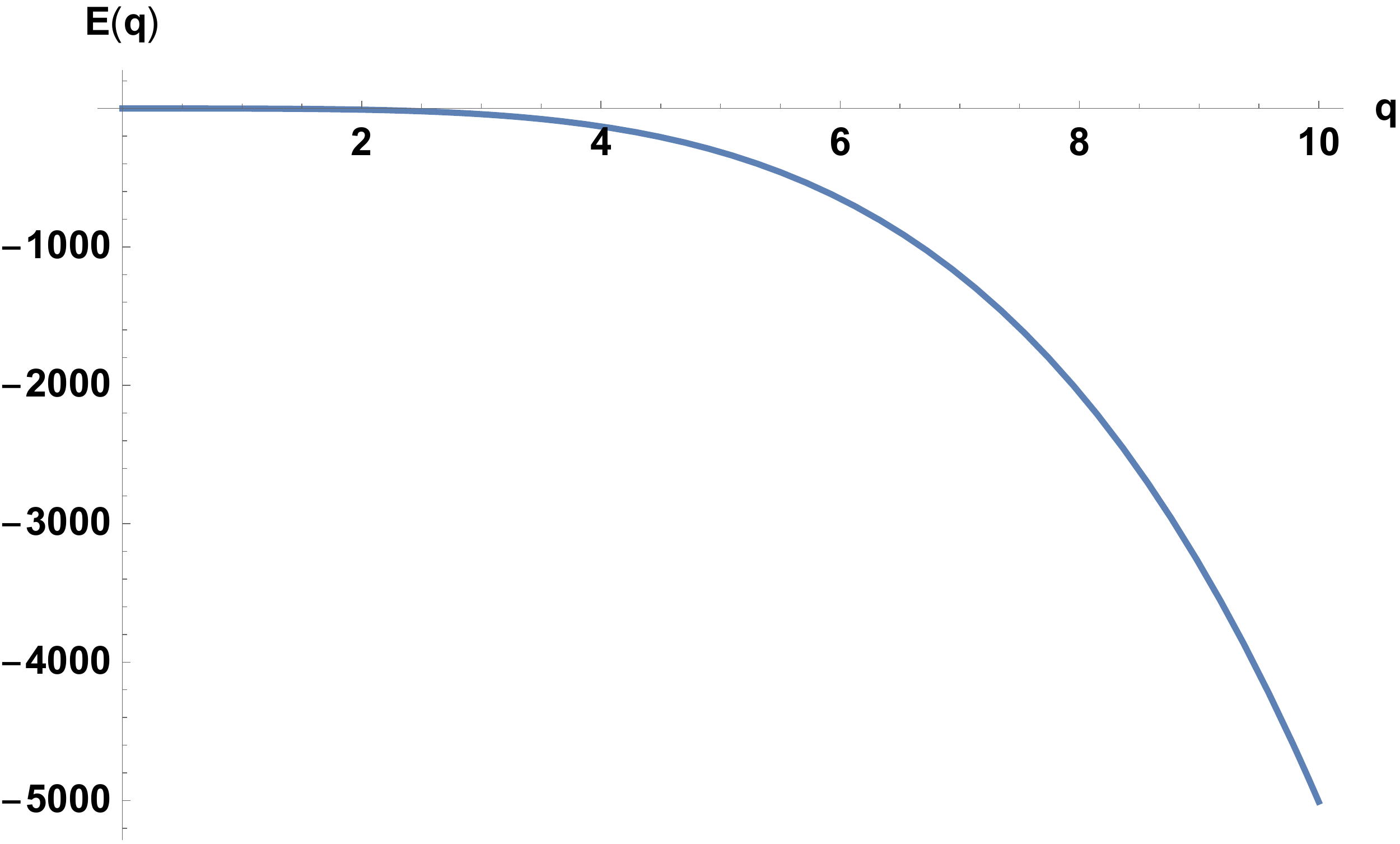}}
\caption{Plots of $\om(q)$ and $\E(q)$ for $q \in \Rp$,  when $\si=1$ and $\be=1$.}
\label{fig-val5}
\end{figure}

\begin{rem}\label{rem-imp}
Let us point out a relevant difference between the focusing and the defocusing case: $\M(\om)$ and $\M(q)$, given by \eqref{eq-massom} and 
\eqref{eq-massq}, respectively, are strictly monotone on their domain with range $\R^+$. In particular, this means that, in the defocusing case,  for every $\mu\in\R^+$, there exists a unique (up to a  phase factor)  standing wave $u_{\om_\mu}$ of mass $\mu$.
\end{rem}
Concerning the stability of these standing waves, one can prove the following

\begin{thm}[Stability in the defocusing case]
 %\label{thm-sdef}
 Let $\sigma\geq1/2$ and $\be>0$. The standing waves defined by \eqref{eq-stan3} are orbitally stable for every $\om\in(0,\wt{\om})$.
\end{thm}

\begin{proof}
 The proof is analogous to the one of Theorem \ref{thm-sfoc}. The main difference is that the key tool now is \cite[Theorem 1]{GSS-87},  instead of \cite[Theorem 3]{GSS-87}. Assumptions 1 and 2 of \cite[Theorem 1]{GSS-87} are the same of \cite[Theorem 3]{GSS-87} and, hence, are easily satisfied, as outlined in the proof of Theorem \ref{thm-sfoc}.
 
 On the other hand, in order to prove the stability of  standing waves, it is sufficient to prove that the operator $H_\om$ defined in \eqref{eq-Hom} satisfies the properties (ii) and (iii). However, arguing exactly as in the proof of Theorem \ref{thm-sfoc}, there results again that
 \[
  H_\om=2\begin{pmatrix} H_{\al_1} +\om & 0 \\ 0 & H_{\al_2} +\om \end{pmatrix}
 \]
 with $H_{\al_1},\,H_{\al_2}$ defined in \eqref{eq-opdelta1} and \eqref{eq-opdelta2} and $\al_1,\,\al_2$ given by \eqref{eq-al12}. Hence, the spectrum of $H_\om$ is given again by \eqref{eq-spectrum}, but now,	 as $\be>0$ and $\om\in(0,\wt{\om})$, there is no negative eigenvalue so that (ii) and (iii) are satisfied and the proof is complete.
\end{proof}

Moreover, in the defocusing case it is possible to give a further characterization of the standing waves, given by the following

\begin{thm}[Ground states in the defocusing case]
 %\label{thm-defmin}
 Let $\be>0$ and $\mu>0$. Then, the energy functional $\E$ restricted to the manifold $V_\mu$ has a unique (up to a phase factor) global minimizer, which is of the form  \eqref{eq-stan3} with $\om=\om_\mu$, where $\om_\mu$ is the unique solution of
 \begin{equation}
  \label{eq-ommu}
  \log(2/\sqrt{\om})-\gamma=2\pi\be(4\pi\om\mu)^{\si}.
 \end{equation}
\end{thm}

\begin{proof}
 Preliminarily, one can see that \eqref{eq-ommu} is equivalent to $\M(\om)=\mu$ with $\M(\om)$ defined by \eqref{eq-massom}. Hence, by Remark \ref{rem-imp}, there is a unique solution $\om_\mu$ for any value of $\mu>0$. It is thus clear that, if a minimizer does exist, then it has to be equal to $u_{\om_\mu}$ up to phase factor.
 
 First, let us fix $\la=\om_\mu$ in \eqref{eq-energy} and in the definition of the norm of $V_\mu$, which is the same of $V$. Consider, therefore, a minimizing sequence $\{\psi_n\}  =\{\phi_{\om_\mu,n}+q_n\G_{\om_\mu}\}  \subset V_\mu$ for $\E$. As $\|\psi_n\|^2=\mu$ and $\be>0$, $\E$ is coercive on $V_\mu$ and hence $\|\psi_n\|_{\om_\mu}\leqslant C$ for every $n$. As a consequence there exists $\psi=\phi_{\om_\mu}+q\,\G_{\om_\mu}\in V$ such that, up to subsequences, 
 \[
  \begin{array}{ll}
   \psi_n\xrightharpoonup[n\to\infty]{w}\psi, & \text{in }L^2(\R^2),\\[.2cm]
   \phi_{\om_\mu,n}\xrightharpoonup[n\to\infty]{w}\phi_{\om_\mu}, & \text{in }H^1(\R^2),\\[.2cm]
   q_n \xrightarrow[n\to\infty]{} q, & \text{in }\C.
  \end{array}
 \]
 Furthermore, by the weak lower semicontinuity of $\E$
 \[
  \E(\psi) \leq \liminf_{n \to +\infty}\E(\psi_n),
 \]
 and, by the weak lower semicontinuity of the norms, $\|\psi\|^2\leq\mu$. Hence, if one can prove that $\|\psi\|^2=\mu$,  the proof is complete.
 
 To this aim, first note that
 \begin{align*}
  \E(\psi) & \geq -\om_\mu\|\psi\|^2+\bigg(\f{\be|q|^{2\si}}{\si+1}+\f{\log(\sqrt{\om_\mu}/2)+\ga}{2\pi}\bigg)|q|^2\\[.2cm]
  & \geq -\om_\mu\mu+\bigg(\f{\be|q|^{2\si}}{\si+1}+\f{\log(\sqrt{\om_\mu}/2)+\ga}{2\pi}\bigg)|q|^2=:f(|q|).
 \end{align*}
 Assuming that $q$ is real-valued (which is not restrictive), one can check that $f$ is minimized for $q=q(\om_\mu)$ and that
 \[
  f\big(q(\om_\mu)\big)=-\f{q^2(\om_\mu)}{4\pi}-\f{\si\be q^{2\si+2}(\om_\mu)}{\si+1}=\E(u_{\om_\mu}).
 \]
 Therefore,
 \[
  \E(\psi)\geq\E(u_{\om_\mu})
 \]
 and, since $\M(u_{\om_\mu})=\mu$, this implies that $u_{\om_\mu}$ is the minimizer of $\E$ on $V_\mu$ up to a phase factor.
\end{proof}

\end{document}